\documentclass{article}
\usepackage{amsmath,amssymb,amsthm}
\usepackage{authblk}
\usepackage{parskip}
\usepackage{tikz}
\usepackage[numbers]{natbib}
\usepackage[margin=2.2cm]{geometry}
\usepackage[bookmarks=false,colorlinks=true,citecolor=blue]{hyperref}

\newtheorem{theorem}{Theorem}{\bfseries}{\itshape}
\newtheorem{definition}{Definition}{\bfseries}{\itshape}
\newtheorem{lemma}{Lemma}{\bfseries}{\itshape}
\newtheorem{corollary}{Corollary}{\bfseries}{\itshape}
\newtheorem{observation}{Observation}{\bfseries}{\itshape}
\newtheorem{claim}{Claim}{\bfseries}{\itshape}
\newtheorem{remark}{Remark}{\bfseries}{\itshape}

\title{\textbf{Acyclic Chromatic Index of Chordless Graphs}}
\author{\normalsize\textbf{Manu~Basavaraju}}
\author{\textbf{Suresh~Manjanath~Hegde}}
\author{\textbf{Shashanka~Kulamarva}}
\affil{National Institute of Technology Karnataka, Surathkal-575025, India\\ Email: \texttt{manub@nitk.edu.in, smhegde@nitk.ac.in, skulamarva.187ma007@nitk.edu.in}}
\date{}

\begin{document}
	\maketitle
	
	\begin{abstract}
		\noindent An acyclic edge coloring of a graph is a proper edge coloring in which there are no bichromatic cycles. The \emph{acyclic chromatic index} of a graph $G$ denoted by $a'(G)$, is the minimum positive integer $k$ such that $G$ has an acyclic edge coloring with $k$ colors. It has been conjectured by Fiam\v{c}\'{\i}k that $a'(G) \le \Delta+2$ for any graph $G$ with maximum degree $\Delta$. Linear arboricity of a graph $G$, denoted by $la(G)$, is the minimum number of linear forests into which the edges of $G$ can be partitioned. A graph is said to be \emph{chordless} if no cycle in the graph contains a chord. Every $2$-connected chordless graph is a \emph{minimally $2$-connected graph}. It was shown by Basavaraju and Chandran that if $G$ is $2$-degenerate, then $a'(G) \le \Delta+1$. Since chordless graphs are also $2$-degenerate, we have $a'(G) \le \Delta+1$ for any chordless graph $G$. Machado, de Figueiredo and Trotignon proved that the chromatic index of a chordless graph is $\Delta$ when $\Delta \ge 3$. They also obtained a polynomial time algorithm to color a chordless graph optimally. We improve this result by proving that the acyclic chromatic index of a chordless graph is $\Delta$, except when $\Delta=2$ and the graph has a cycle, in which case it is $\Delta+1$. We also provide the sketch of a polynomial time algorithm for an optimal acyclic edge coloring of a chordless graph. As a byproduct, we also prove that $la(G) = \lceil \frac{\Delta }{2} \rceil$, unless $G$ has a cycle with $\Delta=2$, in which case $la(G) = \lceil \frac{\Delta+1}{2} \rceil = 2$. To obtain the result on acyclic chromatic index, we prove a structural result on chordless graphs which is a refinement of the structure given by Machado, de Figueiredo and Trotignon for this class of graphs. This might be of independent interest.\\
		
	\end{abstract}
	
	\section{Introduction}
	All graphs considered in this paper are finite and simple. A \emph{path} in a graph $G$ is a sequence of distinct vertices of $G$ such that the consecutive vertices in the sequence are adjacent. A \emph{cycle} in a graph is a path together with an edge between the starting vertex and the ending vertex of the path. A \emph{proper edge coloring} of a graph $G=(V,E)$, with a given set of colors $C$, is a function $c:E \rightarrow C$ such that $c(e_1) \ne c(e_2)$ for any adjacent edges $e_1$ and $e_2$. The minimum number of colors required to perform a proper edge coloring of a given graph $G$ is called the \emph{chromatic index} of $G$ and is denoted by $\chi'(G)$. A \emph{linear forest} is a graph without cycles, in which every vertex has degree at most two. In other words, a linear forest is a disjoint union of paths. A proper edge coloring of a graph $G$ is said to be \emph{acyclic} if there are no bichromatic cycles in $G$ or equivalently if the union of any two color classes induces a linear forest in $G$. The \emph{acyclic chromatic index} (also called as \emph{acyclic edge chromatic number}) of a graph $G$ is the minimum number of colors required to perform an acyclic edge coloring of $G$ and is denoted by $a'(G)$. The concept of acyclic coloring was first introduced by \citet{Grunbaum1973ACP}. The acyclic chromatic index and its vertex analogue can be used to bound other parameters like oriented chromatic number \cite{Kostochka1997AC} and star chromatic number of a graph, both of which have many practical applications including wavelength routing in optical networks \cite{Amar2001Ntwk}. By Vizing’s theorem \cite{Diestel2017GT}, we have $\Delta \le \chi'(G) \le \Delta+1$ where $\Delta=\Delta(G)$ denotes the maximum degree of a vertex in the graph $G$. Since any acyclic edge coloring is also a proper edge coloring, we have $a'(G) \ge \chi'(G) \ge \Delta$.
	
	It has been conjectured by \citet{Fiamvcik1978AC} (and independently by \citet*{Alon2001ACI}) that $a'(G) \le \Delta+2$ for any graph $G$. The best known upper bound for $a'(G)$ for an arbitrary graph till date is $\lceil3.74(\Delta-1)\rceil+1$ given by \citet{Giotis2017AECLLL} which is obtained by using probabilistic techniques. Even though this bound is far from the conjectured bound, the conjecture has been proved for some special classes of graphs. \citet{Alon2001ACI} proved that there exists a constant $k$ such that $a'(G) \le \Delta+2$ for any graph $G$ with girth at least $k\Delta\log\Delta$.
	
	It is known that finding the chromatic index of a graph is NP-Complete even for special classes of graphs. As per \citet{Holyer1981NPCompEC}, it is NP-Complete even for cubic graphs. \citet{Leven1983NPComCI} proved that it is NP-complete to determine whether it is possible to color the edges of a regular graph of degree $k$ with $k$ colors for any $k \ge 3$. Therefore, it is interesting to find out the classes of graphs where the problem is solvable in polynomial time. There are polynomial time algorithms to find the chromatic index of bipartite graphs, series parallel graphs and so on. \citet{Machado2013ECTCChordless} studied the problem in case of chordless graphs.
	
	An edge $e=uv$ of a graph $G$ is said to be a \emph{chord} if the vertices $u$ and $v$ are part of a cycle in $G \setminus e$ which is obtained by deleting the edge $e$ from $G$. A graph is said to be \emph{chordless} if it does not contain a chord. A $2$-connected graph $G$ is said to be \emph{minimally $2$-connected} if $G \setminus e$ is no longer $2$-connected for any edge $e$ in $G$. A chordless graph which is $2$-connected is minimally $2$-connected. The study of structural properties of chordless graphs and minimally $2$-connected graphs can be found in \cite{Dirac1967Min2conn}, \cite{Machado2013ECTCChordless} and \cite{Plummer1968MinBlock}.
	
	\citet{Machado2013ECTCChordless} proved that the edges of any connected chordless graph can be colored using exactly $\Delta$ colors unless $G$ is an odd cycle. They refined the structure of chordless graphs as given by \citet{Leveque2012NoIndK4}, which helped them in proving the result. Since any proper coloring requires at least $\Delta$ colors, this means that the chromatic index of any chordless graph can be determined exactly in polynomial time. \citet{Machado2013ECTCChordless} also gave a polynomial time algorithm to color the given chordless graph with optimum number of colors.
	
	The acyclic chromatic index has been determined exactly for some classes of graphs like outer planar graphs when $\Delta \ne 4$ (\citet{Hou2013AECOuterPlanarErr}, \citet{Hou2010AECOuterPlanar}), series-parallel graphs when $\Delta \ne 4$ (\citet{Wang2011ACIK4MinorFree}), planar graphs with girth at least 5 and $\Delta \ge 19$ (\citet{Basavaraju2011AECPlanar}) and planar graphs with $\Delta \ge 4.2 \times 10 ^{14}$ (\citet{Cranston2019AECPlanar}). In case of outer planar graphs and series-parallel graphs, $a'(G)=\Delta$ if $\Delta \ge 5$ and when $\Delta=3$, they characterize the graphs that require $4$ colors. Note that in case of acyclic edge coloring, at most one color class can be a perfect matching since otherwise the two color classes which are perfect matchings will give rise to at least one cycle when we take their union. This implies that for any regular graph $G$, we have $a'(G) \ge \Delta(G) +1$. Thus any cubic graph requires at least 4 colors for acyclic edge coloring. We also know that any cubic graph can be colored using at most 5 colors.  \citet{Andersen2012AECCubic} proved that any cubic graph other than $K_4$ or $K_{3,3}$ can be colored using 4 colors which determines the acyclic chromatic index of cubic graphs exactly. Further, we note that all the above mentioned results are constructive in nature and hence, they also yield a polynomial time algorithm for the optimum coloring.
	
	A graph $G$ is said to be $k$-degenerate, if every subgraph of $G$ has a vertex of degree at most $k$. Observe that if a graph $G$ is chordless, then since every subgraph of $G$ is also chordless, Lemma~\ref{lem:2-sparseMinimumX} guarantees that every subgraph of $G$ has a vertex of degree at most $2$. Therefore, the class of chordless graphs is a proper subclass of the class of $2$-degenerate graphs. \citet{Basavaraju2010AEC2deg} proved that the edges of any $2$-degenerate graph can be acyclically colored using $\Delta+1$ colors. Therefore, $a'(G) \le \Delta+1$ for a chordless graph $G$. Thus we know that the acyclic chromatic index is either $\Delta$ or $\Delta+1$. 
	
	The acyclic chromatic index of any cubic graph can be determined exactly. But \citet{Alon2002AlgoAC} proved that it is NP-Complete to determine $a'(G)$ even when $\Delta=3$. We can infer that it is NP-Complete to determine $a'(G)$ when $\Delta=3$ and $G$ is not cubic. But these graphs are $2$-degenerate. Thus it is NP-complete to determine $a'(G)$ for $2$-degenerate graphs.
	
	\textbf{Result:}
	
	In this paper, we consider the chordless graphs and prove that the acyclic chromatic index can be determined exactly. This improves the result by \citet{Machado2013ECTCChordless}. To achieve this, we refine the structure of chordless graphs obtained in \cite{Machado2013ECTCChordless}. This result (Lemma \ref{lem:SpecialSplit})  might be of independent interest where we need a more refined structure than that is required for proper edge coloring. In particular, we prove the following theorem.
	
	\begin{theorem}\label{thm:ACIChordless}
		Let $G$ be a connected chordless graph with maximum degree $\Delta$. Then $a'(G)=\Delta$, unless $G$ is a cycle $($in which case $a'(C_n)=\Delta+1=3)$.
	\end{theorem}
	
	\begin{remark}\label{rem:ConnComp}
		Since the acyclic edge coloring of the components of a graph can be easily extended to the coloring of the whole graph, it is enough to prove the statement for connected graphs. Hence, we have the following Corollary.
	\end{remark}
	
	\begin{corollary}
		Let $G$ be a chordless graph with maximum degree $\Delta$. Then $a'(G)=\Delta$, unless $G$ has a cycle with $\Delta=2$ $($in which case $a'(G)=\Delta+1=3)$.
	\end{corollary}
	
	\begin{remark}
		Linear arboricity of a graph $G$ denoted by $la(G)$, is the minimum positive integer $k$ such that the edges of $G$ can be partitioned into $k$ linear forests. It is conjectured by \citet{Akiyama1980LA} that $la(G) \le \lceil \frac{\Delta + 1 }{2} \rceil$. Since the union of any two color classes in acyclic edge coloring is a linear forest, we can infer that $la(G) \le \lceil \frac{a'(G)}{2} \rceil$. By a result of \citet{Basavaraju2010AEC2deg}, the conjecture is true for $2$-degenerate graphs. Our result in this paper shows that the linear arboricity of chordless graphs can be determined exactly, i.e., for a chordless graph $G$, we have $la(G) = \lceil \frac{ \Delta}{2} \rceil$ unless $\Delta=2$ and $G$ contains a cycle, in which case $la(G) = \lceil \frac{\Delta+1}{2} \rceil = 2$. Thus we have the following Corollary.
	\end{remark}
	
	\begin{corollary}
		If $G$ is a chordless graph with maximum degree $\Delta$, then $la(G) = \lceil \frac{\Delta }{2} \rceil$, unless $G$ has a cycle with $\Delta=2$ $($in which case $la(G) = \lceil \frac{\Delta+1}{2} \rceil = 2)$.
	\end{corollary}
	
	\begin{remark}
		From Theorem~\ref{thm:ACIChordless} and Remark~\ref{rem:ConnComp}, we can infer that $a'(G)$ can be determined exactly for chordless graphs in polynomial time. We also give a sketch of a polynomial time algorithm to acyclically color the edges of a given chordless graph with optimum number of colors.
	\end{remark}
	
	\section{Preliminaries}
	Let $G=(V,E)$ be a simple, finite and undirected graph with $n$ vertices and $m$ edges. The \emph{degree} of any vertex $x$ in $G$, represented as $deg_G(x)$ or simply $deg(x)$ (if $G$ is understood from the context), is the number of edges in $G$ which are incident on the vertex $x$. A vertex with degree zero is said to be an \emph{isolated vertex}. Maximum degree and minimum degree of $G$, represented as $\Delta(G)$ and $\delta(G)$ respectively or simply $\Delta$ and $\delta$ (if the graph $G$ is obvious), are the maximum degree and minimum degree of vertices in $G$ respectively. For any vertex $x \in V$, $N_G(x)$ is the set of all neighbors of $x$ in $G$. Throughout the paper, $N_G(x)$ is written as $N(x)$ whenever $G$ is understood from the context. Let $R$ be a subset of $E$, i.e., $R \subseteq E$. Then $G \setminus R$ is the subgraph of $G$ obtained by the vertex set $V$ and the edge set $E \setminus R$. Let $S \subseteq V$. Then $G \setminus S$ is the subgraph of $G$ obtained by the vertex set $V \setminus S$ and the edge set $E \setminus \{e \in E \text{ / } \exists v \in S \text{ such that } e \text{ is incident on } v\}$. If either $R$ or $S$ is a singleton set $\{x\}$, then we write $G \setminus \{x\}$ as $G \setminus x$ throughout the paper.
	
	A vertex $x \in V$ is said to be a \emph{cut vertex} in $G$ if $G \setminus x$ is disconnected. A graph $G$ is said to be \emph{$k$-connected} if the removal of less than $k$ vertices does not disconnect $G$. Let $S \subseteq V$. Then $G[S]$ is the subgraph of $G$ induced by the vertices in $S$, i.e., $G[S]=(V',E')$ where $V'=S$ and $E' = \{ab \in E \text{ / } a \in S, b \in S\}$. Let $e=xy$ be an edge in $G$. Then the process of deleting the vertices $x$ and $y$ from $G$ and replacing it by a new vertex $z$ such that $N(z) = (N(x) \setminus y) \cup (N(y) \setminus x)$ is called as \emph{edge contraction} corresponding to the edge $e$. If the operation results in a multi edge, we keep only one edge between those vertices. The number of edges in a path is said to be the \emph{length} of the path. Distance between any two vertices $x$ and $y$ in $G$ is the number of edges in the shortest $(x,y)$-path. An odd cycle in a graph is a closed path of odd length. Graph $G$ is said to be a \emph{bipartite graph}, if the vertex set $V$ can be partitioned into two sets $X$ and $Y$ such that every edge in $G$ has one end in $X$ and the other end in $Y$. It is proved that a graph is bipartite if and only if it does not have an odd cycle in it. See \cite{West2001IGT} for further notations and definitions. Throughout the paper, whenever just the word coloring is mentioned, it should be taken as acyclic edge coloring.
	
	The following definitions and lemmas are given by \citet{Basavaraju2010AEC2deg}. We will mention the ones which are required for our discussion below.
	
	\begin{definition}[\cite{Basavaraju2010AEC2deg}]
		An edge coloring $c$ of a subgraph $H$ of a graph $G$ is called a \textbf{partial edge coloring} of $G$.
	\end{definition}
	
	Note that $H$ can be $G$ itself. Therefore, an edge coloring of $G$ is also a partial edge coloring of $G$. Given partial edge coloring $c$ of $G$ is proper (and acyclic) if it is proper (and acyclic) in the corresponding subgraph $H$.  Note that $c(e)$ may not be defined for an edge $e$ with respect to a partial coloring $c$. So, whenever we use $c(e)$, we are considering an edge $e$ for which $c(e)$ is defined, even though it is not explicitly mentioned. Let $c$ be a partial edge coloring of $G$. For any vertex $u \in V$, we define $F_u(c) = \{c(uv) \text{ / } v \in N_G(u)\}$. For an edge $ab \in E$, we define $F_{ab}(c) = F_b(c) \setminus \{c(ab)\}$. We will abbreviate the notation as $F_u$ and $F_{ab}$ when the coloring $c$ is obvious from the context. Note that $F_{ab}$ is different from $F_{ba}$.
	
	\begin{definition}[\cite{Basavaraju2010AEC2deg}]
		An $(\alpha,\beta)$-maximal bichromatic path with respect to a partial coloring $c$ of $G$ is a maximal path in $G$ consisting of edges that are colored using the colors $\alpha$ and $\beta$ alternatingly. An $(\alpha,\beta,a,b)$-\textbf{maximal bichromatic path} is an $(\alpha,\beta)$-maximal bichromatic path which starts at the vertex $a$ with an edge colored with $\alpha$ and ends at the vertex $b$.
	\end{definition}
	
	The following lemma (mentioned as a fact in \cite{Basavaraju2010AEC2deg}) follows from the definition of acyclic edge coloring. We implicitly assume this lemma throughout the paper.
	
	\begin{lemma}[\cite{Basavaraju2010AEC2deg}]
		Given a pair of colors $\alpha$ and $\beta$ of a proper coloring $c$ of $G$, there can be at most one $(\alpha,\beta)$-maximal bichromatic path containing a particular vertex $v$ in $G$, with respect to the coloring $c$.
	\end{lemma}
	
	\begin{definition}[\cite{Basavaraju2010AEC2deg}]
		An $(\alpha,\beta,a,b)$-maximal bichromatic path in a graph $G$, which ends at $b$ with an edge colored $\alpha$, is said to be an $(\alpha,\beta,ab)$-\textbf{critical path} if the vertices $a$ and $b$ are adjacent in $G$.
	\end{definition}
	
	\begin{definition}[\cite{Basavaraju2010AEC2deg}]
		Let $c$ be a partial coloring of $G$. Let $u,i,j \in V(G)$ and $ui,uj \in E(G)$. A \textbf{color exchange} with respect to the edges $ui$ and $uj$ is defined as the modification of the current partial coloring $c$ by exchanging the colors of the edges $ui$ and $uj$ to get a partial coloring $c'$, i.e., $c'(ui)=c(uj)$, $c'(uj)=c(ui)$ and for all other edges $e$ in $G$, $c'(e)=c(e)$. The color exchange with respect to the edges $ui$ and $uj$ is said to be proper if the coloring obtained after the exchange is proper. The color exchange with respect to the edges $ui$ and $uj$ is \textbf{valid} if and only if the coloring obtained after the exchange is acyclic.
	\end{definition}
	
	A color $\alpha$ is said to be a \emph{candidate} for an edge $e$ in $G$ with respect to a partial coloring $c$ of $G$ if none of the adjacent edges of $e$ are colored $\alpha$. A candidate color $\alpha$ is said to be \emph{valid} for an edge $e$ if assigning the color $\alpha$ to $e$ does not result in any bichromatic cycle in $G$. Following lemma is mentioned as a fact by \citet{Basavaraju2010AEC2deg}, since it is obvious.
	
	\begin{lemma}[\cite{Basavaraju2010AEC2deg}]\label{lem:ColorValidity}
		Let $c$ be a partial coloring of $G$. A candidate color $\beta$ is not valid for an edge $e=(a,b)$ if and only if there exists $\alpha \in F_{ba} \cap F_{ab}$ such that there is an $(\alpha,\beta,ab)$-critical path in $G$ with respect to the coloring $c$.
	\end{lemma}
	
	\section{Structure of Chordless Graphs}
	We will start with a strict subclass of chordless graphs before actually moving into the structure of chordless graphs.
	
	\begin{definition}
		A graph $G$ is \emph{$2$-sparse} if every edge of $G$ is incident on at least one vertex of degree at most $2$.
	\end{definition}
	
	Observe that a $2$-sparse graph is also chordless since a chord of a cycle in a graph is an edge whose end vertices have degree at least 3.
	
	\begin{definition}
		A \emph{proper $2$-cutset} of a connected graph $G=(V,E)$ is a pair of non-adjacent vertices $(a,b)$ such that $V$ can be partitioned into non-empty sets $X$, $Y$ and $\{a,b\}$ so that there is no edge between any vertex in $X$ and any vertex in $Y$ and both $G[X \cup \{a,b\}]$ and $G[Y \cup \{a,b\}]$ contain an $ab$-path but neither $G[X \cup \{a,b\}]$ nor $G[Y \cup \{a,b\}]$ is an induced path. Then $(X,Y,a,b)$ is called a \emph{split} of the proper $2$-cutset $(a,b)$. The block $G_X(a,b)$ $($respectively $G_Y(a,b))$ is the graph obtained by taking $G[X \cup \{a,b\}]$ $($respectively $G[Y \cup \{a,b\}])$ and adding a new vertex $w$ called as the marker vertex, adjacent to both $a$ and $b$. 
	\end{definition}
	
	The structure of chordless graphs and minimally $2$-connected graphs were studied by \citet{Dirac1967Min2conn} and \citet{Plummer1968MinBlock}. Recently \citet{Leveque2012NoIndK4} gave another structural property of chordless graphs. In particular, they proved the following lemma.
	
	\begin{lemma}[\cite{Leveque2012NoIndK4}]\label{lem:2-sparseOR2-cutset}
		If $G$ is a $2$-connected chordless graph, then either $G$ is $2$-sparse or $G$ admits a proper $2$-cutset.
	\end{lemma}
	
	This is a useful property in proving many results on chordless graphs. But when \citet{Machado2013ECTCChordless} tried to study the edge coloring of chordless graphs, they needed a refined structure with respect to the proper $2$-cutsets that are mentioned in the above statement. Hence, they proved the following lemma.
	
	\begin{lemma}[\cite{Machado2013ECTCChordless}]\label{lem:2-sparseMinimumX}
		Let $G$ be a $2$-connected, not $2$-sparse chordless graph. Let $(X,Y,a,b)$ be a split of a proper $2$-cutset of $G$ such that $|X|$ is minimum among all possible such splits. Then both $a$ and $b$ have at least two neighbors in $X$ and $G_X(a,b)$ is $2$-sparse.
	\end{lemma}
	
	This structural result helped them in proving that the edges of every chordless graph with maximum degree at least 3 can be colored using $\Delta$ colors. The split helps in the inductive proof. The main idea is to extend the coloring of $G_Y(a,b)$ to the coloring of $G$ by exploiting the structure of $G_X(a,b)$ which is $2$-sparse.
	
	In this paper, we also use the induction method. The proper $2$-cutset and $G_X(a,b)$ play a major role in choosing the smaller subgraph. But since the coloring has to be acyclic, we need to limit the cycles to one block where we can handle them. We try to limit the possible cycles within $G_X(a,b) \setminus w$ since it is $2$-sparse. But the major hurdle is when all the edges in the $2$-sparse side are incident on the vertices $a$ and $b$. In proper edge coloring we can give the missing colors on those vertices to the edges in $G_X(a,b) \setminus w$ by appropriately permuting the colors if needed. But that does not work with respect to acyclic edge coloring since there is a possibility of cycles being created in the graph even when we permute the colors. Hence we need a much more refined structure with respect to the proper $2$-cutset than what is required by \citet{Machado2013ECTCChordless}.
	
	Note that when all the edges are incident to either $a$ or $b$, then the graph $G_X(a,b)$ is isomorphic to a complete bipartite graph $K_{2,t}$ for $t \ge 2$. When $t=2$, the graph $G_X(a,b)$ is a $4$-cycle. We prove the following lemma.
	
	\begin{lemma}\label{lem:SpecialSplit}
		Let $G$ be a $2$-connected, not $2$-sparse chordless graph. Then there exists a split $(X,Y,a,b)$ in $G$ with the following properties.
		\begin{enumerate}
			\item $G_X(a,b)$ is $2$-sparse.
			\item $G_X(a,b)$ is not isomorphic to $K_{2,t}$ for any $t \ge 3$.
			\item In $G_X(a,b)$, we have $deg(a) \ge 3$ and either $deg(b) \ge 3$ or $deg(b)=2$ with $X$ being minimal.
		\end{enumerate}
	\end{lemma}
	
	\begin{proof}
		Since $G$ is a $2$-connected, not $2$-sparse chordless graph, the existence of splits with the property $(i)$ follows from Lemma~\ref{lem:2-sparseOR2-cutset} and Lemma~\ref{lem:2-sparseMinimumX}. We define $S(G)$ to be the set of all splits $(X,Y,a,b)$ of any proper $2$-cutset of $G$ such that $G_X(a,b)$ is $2$-sparse. A proper $2$-cutset $(a,b)$ of $G$ is said to be an \emph{isolating pair}, if there exists a split $(X,Y,a,b)$ such that $G_X(a,b)$ is isomorphic to $K_{2,t}$ for some $t \ge 3$. The set of all isolating pairs of $G$ is represented as $I(G)$. For each $(a,b) \in I(G)$, we define the corresponding set of all degree $2$ neighbors denoted by $N_G^2(a,b)$, as follows.
		
		\begin{equation*}
			N_G^2(a,b) = \{x \in V(G) \text{ / } N(x) = \{a,b\}\}
		\end{equation*}
		
		Note that $|N_G^2(a,b)| \ge 2$ for any $(a,b) \in I(G)$. For each $(a,b) \in I(G)$, we arbitrarily select a vertex from $N_G^2(a,b)$ as a representative vertex of the pair $(a,b)$ and denote it as $r_{ab}$. An \emph{ip-deleted subgraph} of $G$ is the graph obtained from $G$ by deleting all the vertices in $N_G^2(u,v)$ other than $r_{uv}$ for each isolating pair $(u,v) \in I(G)$. We make the following observations about the subgraph $H$.
		
		\begin{observation}\label{obs:UniqueRepresentative}
			If $(u,v) \in I(G)$, then $r_{uv}$ is the unique degree $2$ vertex in the set $N_H(u) \cap N_H(v)$.
		\end{observation}
		
		\begin{observation}\label{obs:DegreeDecrease}
			The only vertices in $H$ whose degrees decrease with respect to their degree in $G$ are those which are part of an isolating pair in $G$.
		\end{observation}
		
		\begin{observation}\label{obs:DegreeTwoVertexRemoval}
			While obtaining $H$ from $G$, all the vertices that are removed from $G$ have degree $2$ in $G$.
		\end{observation}
		
		Now, to prove the property $(ii)$ of the lemma, it is enough to prove that there exists a split $(X,Y,a,b) \in S(G)$ such that $G_X(a,b)$ is not isomorphic to $K_{2,t}$ for any $t \ge 3$. If such a split exists then we get a split satisfying the property $(ii)$ of the lemma. Otherwise every split $(X,Y,a,b) \in S(G)$ is such that $G_X(a,b)$ is isomorphic to $K_{2,t}$ for some $t \ge 3$. Since $G$ is $2$-connected, we have $\delta(G) \ge 2$. Let $H$ be the ip-deleted subgraph of $G$. Since $\delta(G) \ge 2$, it is clear that $\delta(H) \ge 2$.
		
		\begin{claim}
			There exists no isolating pair in $H$, i.e., $I(H)=\phi$.
		\end{claim}
		
		\begin{proof}
			By way of contradiction, assume that there exists an isolating pair $(u,v)$ in $H$ which implies the existence of a split $(X,Y,u,v)$ in $H$ such that $H_X(u,v)$ is isomorphic to $K_{2,t}$ for some $t \ge 3$. Since $(u,v) \in I(H)$, we have that $deg_H(u) \ge 3$ and $deg_H(v) \ge 3$. Since $|N_H^2(u,v)| \ge 2$, $N_H(u) \cap N_H(v)$ contains at least two degree 2 vertices. Suppose $(u,v) \in I(G)$. Then by Observation~\ref{obs:UniqueRepresentative}, $r_{uv}$ is the unique degree 2 vertex in $N_H(u) \cap N_H(v)$, a contradiction to the fact that $|N_H^2(u,v)| \ge 2$. Therefore, we infer that $(u,v) \notin I(G)$. Hence, there exists a vertex $z \in N_H^2(u,v)$ such that $deg_G(z) \ge 3$. But $deg_H(z)=2$. Hence by Observation~\ref{obs:DegreeDecrease}, $z$ is part of an isolating pair, say $(z,w)$ in $G$. Note that $r_{zw} \in V(H)$ and $zr_{zw} \in E(H)$. But we know that $N_{H}(z)=\{u,v\}$. Hence, we can infer that $r_{zw} \in \{u,v\}$. But $deg_{H}(r_{zw})=2$, a contradiction since $deg_H(u) \ge 3$ and $deg_H(v) \ge 3$. Thus we can infer that $I(H)=\phi$ and hence the claim holds.
		\end{proof}
		
		Further, we claim that there exists a degree 2 vertex in $H$ which has a higher degree in $G$.
		
		\begin{claim}\label{clm:HigherDegreeVertex}
			There exists an edge $ac \in E(G)$ satisfying $deg_H(a)=2$, $deg_G(a)>2$ and $deg_G(c)>2$.
		\end{claim}
		
		\begin{proof}
			
			Suppose $H$ is $2$-sparse. Since $G$ is not $2$-sparse, there exists an edge $ac \in E(G)$ such that $deg_G(a)>2$ and $deg_G(c)>2$. Thus by Observation~\ref{obs:DegreeTwoVertexRemoval}, we infer that $a,c \in V(H)$. Since $H$ is $2$-sparse, one of these vertices, either $a$ or $c$ should have degree $2$ in $H$. Without loss of generality let it be $a$. Therefore, we have the desired edge $ac$ satisfying the statement of the claim.
			
			On the other hand suppose $H$ is not $2$-sparse. Then since $I(H)=\phi$, by Lemma~\ref{lem:2-sparseMinimumX}, we obtain a split $(\tilde{X},\tilde{Y},u,v)$ in $S(H)$ such that $G_{\tilde{X}}(u,v)$ is not isomorphic to $K_{2,t}$ for any $t \ge 3$. Since every split $(X,Y,a,b) \in S(G)$ is such that $G_X(a,b)$ is isomorphic to $K_{2,t}$ for some $t \ge 3$, we have $(\tilde{X},\tilde{Y},u,v) \notin S(G)$ which implies that $G_{\tilde{X}}(u,v)$ is not $2$-sparse. Hence, there exists an edge $ac \in E(G_{\tilde{X}}(u,v))$ such that $deg_{G_{\tilde{X}}(u,v)}(a)>2$ (therefore $deg_G(a)>2$) and $deg_{G_{\tilde{X}}(u,v)}(c)>2$ (therefore $deg_G(c)>2$). Therefore, by Observation~\ref{obs:DegreeTwoVertexRemoval}, we infer that $a,c \in V(H)$. Further, $H_{\tilde{X}}(u,v)$ is $2$-sparse since $(\tilde{X},\tilde{Y},u,v) \in S(H)$. Therefore, one of these vertices, either $a$ or $c$ should have degree $2$ in $H_{\tilde{X}}(u,v)$ (also in $H$). Without loss of generality let it be $a$. Hence, we have the desired edge $ac$ satisfying the statement of the claim.
			
			Thus in any case, there exists an edge $ac \in E(G)$ satisfying $deg_H(a)=2$, $deg_G(a)>2$ and $deg_G(c)>2$.
		\end{proof}
		
		By Claim~\ref{clm:HigherDegreeVertex}, there exists a degree 2 vertex $a$ in $H$ which has a higher degree in $G$. By Observation~\ref{obs:DegreeDecrease}, the vertex $a$ was part of an isolating pair, say $(a,b)$ in $G$ with a corresponding split $(X,Y,a,b)$. Since $(a,b) \in I(G)$, $X=N_G^2(a,b)$ and $Y = V(G) \setminus (\{a,b\} \cup N_G^2(a,b))$. We know that $r_{ab} \in H$. Since $deg_G(c)>2$ and $ac \in E(G)$, we infer that $c \in Y$. Since $deg_H(a)=2$ and $r_{ab} \in H$, $c$ is the unique neighbor of $a$ in $Y$. Hence $N_H(a)=\{c,r_{ab}\}$. Now consider the vertex pair $(b,c)$. If $bc \in E(G)$, then either $c$ is a cut vertex or $bc$ is a chord, a contradiction since $G$ is a $2$-connected chordless graph. Thus $bc \notin E(G)$. Now, consider a split $(X',Y',b,c)$ where $X' = X \cup \{a\}$ and $Y' = Y \setminus c$. We claim that $G_{X'}(b,c)$ is $2$-sparse. Since $G_X(a,b)$ is $2$-sparse, it is enough to prove that $deg_{G_{X'}(b,c)}(c)=2$. Since $c$ is the unique neighbor of $a$ in $Y$ and $bc \notin E(G)$, we have that $N_{G}(c) \cap X' = \{a\}$. Thus, $deg_{G_{X'}(b,c)}(c)=2$ which implies that $G_{X'}(b,c)$ is $2$-sparse. Therefore by definition, $(X',Y',b,c) \in S(G)$. Notice that $b \notin N_{G_{X'}(b,c)}(a)$, which implies that $G_{X'}(b,c)$ is not isomorphic to $K_{2,t}$ for any $t \ge 3$. But this is a contradiction to the fact that every split in $S(G)$, in particular $(X',Y',b,c)$ is such that $G_X'(b,c)$ is isomorphic to $K_{2,t}$ for some $t \ge 3$. Hence, we can conclude that there exists a split $(X,Y,a,b) \in S(G)$ such that $G_X(a,b)$ is not isomorphic to $K_{2,t}$ for any $t \ge 3$. This includes the property $(ii)$ of the Lemma into $S(G)$.
		
		Now, we define $S'(G)$ to be the set of all splits $(X,Y,a,b)$ of any proper $2$-cutset of $G$ such that $G_X(a,b)$ is $2$-sparse and is not isomorphic to $K_{2,t}$ for any $t \ge 3$. Consider a split $s=(X,Y,a,b)$ in $S'(G)$ such that $X$ is minimal. Since $G$ is a chordless graph, every $(a,b)$-path in $G_X(a,b)$ is an induced path. Hence, an $(a,b)$-path $P$ of maximum length in $G_X(a,b)$ is also an induced path. Since $(X,Y,a,b) \in S'(G)$, $P$ has at least $3$ edges, otherwise the corresponding $G_X(a,b)$ is isomorphic to $K_{2,t}$ for some $t \ge 3$, a contradiction. Hence, there exist $x,y \in X$ such that $x$ is adjacent to $a$ but not $b$ and $y$ is adjacent to $b$ but not $a$. Note that both $a$ and $b$ have at least one neighbor each in $X$ and at least one neighbor each in $Y$ as per the definition of a proper $2$-cutset.
		
		First we claim that at least one among $a$ or $b$ has at least two neighbors in $X$. By way of contradiction, assume that both $a$ and $b$ have unique neighbors in $X$. Since $G[X \cup \{a,b\}]$ is not an induced path in $G$ which is a chordless graph, the unique neighbors of $a$ and $b$ in $X$ are two distinct non adjacent vertices in $G$. This implies that $x$ and $y$ are the unique neighbors of $a$ and $b$ in $X$ respectively. Now, consider the split $(X',Y',x,b)$ where $X' = X \setminus x$ and $Y' = Y \cup \{a\}$. Since $deg_{G_X(a,b)}(b)=deg_{G_{X'}(x,b)}(b)=2$, $G_{X'}(x,b)$ is not isomorphic to $K_{2,t}$ for any $t \ge 3$. Moreover, $G_{X'}(x,b)$ is $2$-sparse, since $G_X(a,b)$ is $2$-sparse. Thus we can infer that $(X',Y',x,b) \in S'(G)$, a contradiction to the minimality of $s$ in $S'(G)$. Therefore, we have that at least one among $a$ or $b$ has at least two neighbors in $X$, implying that either $deg(a) \ge 3$ or $deg(b) \ge 3$ in $G_X(a,b)$. Without loss of generality let $deg(a) \ge 3$ in $G_X(a,b)$.
		
		If $deg(b) \ge 3$ in $G_X(a,b)$, then we are done. Otherwise let $deg(b)=2$ in $G_X(a,b)$. This implies that $y$ is the unique neighbor of $b$ in $X$. Now, consider the split $(X'',Y'',a,y)$ where $X'' = X \setminus y$ and $Y'' = Y \cup \{b\}$. If $G_{X''}(a,y)$ is not isomorphic to $K_{2,t}$ for any $t \ge 3$, then since $G_{X''}(a,y)$ is $2$-sparse (as $G_X(a,b)$ is $2$-sparse), we have $(X'',Y'',a,y) \in S'(G)$, a contradiction to the minimality of $s$ in $S'(G)$. Therefore, $G_{X''}(a,y)$ is isomorphic to $K_{2,t}$ for some $t \ge 3$, implying the minimality of $X$ as desired.
		
		This concludes the proof of the lemma.
	\end{proof}
	
	\section{Proof of Theorem~\ref{thm:ACIChordless}}\label{prf:ACIChordlessThm}
	\begin{proof}
		Let $G$ be the given connected chordless graph with $n$ vertices and $m$ edges. If $\Delta=1$, then $G$ has only one edge which requires only one color for any coloring, indicating that $a'(G)=1=\Delta$. If $\Delta=2$ and $G$ is acyclic, then $G$ is a path. We know that any path can be acyclically edge colored using 2 colors. If $\Delta=2$ and $G$ is not acyclic, then $G$ is a cycle. Note that for the acyclic coloring of the edges of a cycle, we need at least 3 colors by the definition of acyclic edge coloring. It is easy to see that 3 colors are also sufficient. Therefore, $a'(G)=3=\Delta+1$. Hence we assume that $\Delta \ge 3$. Further, we use induction on $m$ in the proof, when necessary.
		
		If G is not $2$-connected, then there exists a vertex $x \in V$ which is a cut vertex. Let $C_1, C_2,...,C_k$ be the components in $G \setminus x$. For each $i \in \{1,2,....,k\}$ we define $C'_i$ as $G[V(C_i) \cup x]$. Since each $C'_i$ is chordless and has less than $m$ edges, by Induction Hypothesis (I.H.), we acyclically color all $C'_i$'s with $\Delta$ colors since $\Delta(C'_i) \le \Delta, \forall i$. Let this coloring be $c'$. Now, we extend $c'$ to a coloring $c$ of $G$. Observe that the coloring of each $C'_i$ is independent of each other except for the edges incident on $x$. Now permute the colors in each $C'_i$ so that the edges incident on $x$ receive different colors, to get the coloring $c$ of $G$. It is easy to see that the coloring $c$ is proper and acyclic and hence we are done. Thus, we also assume that $G$ is $2$-connected. This also implies that $\delta(G) \ge 2$.
		
		If $G$ has an edge $uv$ whose end vertices have degree $2$, then we obtain a graph $H$ from $G$ by contracting the edge $uv$ to a new vertex $k_{uv}$. Let $u'$ be the neighbor of $u$ other than $v$ and let $v'$ be the neighbor of $v$ other than $u$ in $G$. Note that $H$ is chordless and has less than $m$ edges. Hence by I.H., $H$ can be colored using $\Delta$ colors, since $\Delta(H) \le \Delta$. Let $d$ be one such coloring. Now we extend the coloring $d$ to a coloring $c$ of $G$. Assign $c(uu')=d(k_{uv}u')$, $c(vv')=d(k_{uv}v')$ and assign a color other than these two colors, to the edge $uv$ (we have at least three colors since $\Delta \ge 3$). For any other edge $e$ in $G$, $c(e)=d(e)$. It is easy to see that the coloring $c$ is acyclic and we are done. Hence, we also assume that $G$ does not have an edge whose end vertices have degree $2$.
		
		The following observations and the subsequent lemma are used multiple times further down the proof.
		
		\begin{observation}\label{obs:2sparseBipartite}
			Let $G$ be a $2$-connected, $2$-sparse graph such that no edge is incident on two degree $2$ vertices. Then $G$ is bipartite.
		\end{observation}
		
		\begin{proof}
			Since $G$ is $2$-sparse and $2$-connected, we have $\delta(G)=2$. Also since $G$ is $2$-sparse, the set of vertices $\{x \in V(G) \text{ / } deg(x) \ge 3\}$ form an independent set. Now, since there is no edge between any two vertices of degree 2, the set of degree 2 vertices also form an independent set. Therefore, $G$ is bipartite.
		\end{proof}
		
		\begin{observation}\label{obs:PathLengthEven}
			Let $G$ be a $2$-connected, not $2$-sparse chordless graph such that no edge is incident on two degree $2$ vertices and let $(X,Y,a,b)$ be a split of $G$ such that $G_X(a,b)$ is $2$-sparse, $deg(a) \ge 3$ and $deg(b) \ge 3$. Then any $(a,b)$-path in $G_X(a,b)$ is of even length.
		\end{observation}
		
		\begin{proof}
			Since $G$ is $2$-connected and not $2$-sparse, by Lemma~\ref{lem:2-sparseMinimumX} we have a split $(X,Y,a,b)$ of $G$ such that $G_X(a,b)$ is $2$-sparse. Since there is no edge incident on two degree 2 vertices in $G$ and $deg(a) \ge 3$ and $deg(b) \ge 3$, there is no edge incident on two degree 2 vertices in $G_X(a,b)$ as well. Hence by Observation~\ref{obs:2sparseBipartite}, $G_X(a,b)$ is bipartite with $a$ and $b$ on the same side of the bipartition. Hence, any $(a,b)$-path should be of even length.
		\end{proof}
		
		\begin{lemma}\label{lem:BichromaticPathColorExchange}
			Let $d$ be a partial acyclic edge coloring of a graph $G$ using at most $\Delta$ colors. Let $P = v_1 v_2 \cdots v_k$ be a maximal bichromatic path with $d(v_{2i-1}v_{2i})=\alpha$ and $d(v_{2i}v_{2i+1})=\beta$ for $1 \le i \le \lfloor \frac{k-1}{2} \rfloor$. For all $v_{2i-1} \in V(P)$, let $deg(v_{2i-1})=2$ and for all $v_{2i} \in V(P)$, let the neighbors of $v_{2i}$ are all degree $2$ vertices. Let $N(v_1)=\{v_2,s\}$, $deg(s) \ge 3$ and $s \notin V(P)$. Further, let the edge $v_1s$ be not colored with respect to the coloring $d$. Then there exists a valid partial acyclic edge coloring $c$ of $G$ using at most $\Delta$ colors such that all the edges colored in $d$ and the edge $v_1s$ are colored in $c$.
		\end{lemma}
		
		\begin{proof}
			If there is a valid color for the edge $sv_1$, then we assign the same to $sv_1$ to get the required coloring $c$. From now on, we assume that there is no valid color for the edge $v_1s$. Hence, either there is no candidate color for the edge $sv_1$ which implies that $|F_s \cup F_{v_1}| = \Delta$ and $\alpha \notin F_{v_1s}$ or no candidate color, say $\eta$ is valid for the edge $sv_1$ which implies that $\alpha \in F_{v_1s}$ and there exists an $(\alpha,\eta,v_1s)$-critical path in $G$ with respect to the coloring $d$. Note that in the latter case $\eta \ne \beta$ since the $(\alpha,\beta)$-bichromatic path $P$ does not reach the vertex $s$.
			
			We obtain a coloring $c'$ from $d$, by exchanging the colors $\alpha$ and $\beta$ along the path $P$ so that $F_{sv_1}(c')=\{\beta\}$ and $F_{v_1s}(c')=F_{v_1s}(d)$. If there is no bichromatic cycle created by this exchange, then we let coloring $c''=c'$.
			
			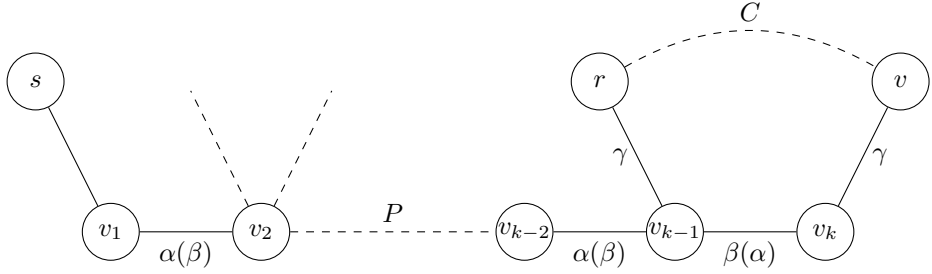
\begin{figure}[h]
				\centering
				\begin{tikzpicture}
					\begin{scope}[every node/.style={circle,draw,inner sep=0pt, minimum size=5ex}]
						\node (s) at (0,2) {$s$};
						\node (v1) at (1,0) {$v_1$};
						\node (v2) at (3,0) {$v_2$};
						\node (vk-2) at (6.5,0) {$v_{k-2}$};
						\node (vk-1) at (8.5,0) {$v_{k-1}$};
						\node (vk) at (10.5,0) {$v_k$};
						\node (r) at (7.5,2) {$r$};
						\node (v) at (11.5,2) {$v$};
					\end{scope}
					\node (x) at (2,2) {};
					\node (y) at (4,2) {};
					\draw (s) to (v1);
					\draw (v1) to node[below]{$\alpha (\beta)$} (v2);
					\draw (vk-2) to node[below]{$\alpha (\beta)$} (vk-1);
					\draw (vk-1) to node[below]{$\beta (\alpha)$} (vk);
					\draw (vk-1) to node[left]{$\gamma$} (r);
					\draw (vk) to node[right]{$\gamma$} (v);
					\draw [dashed] (v2) to node[above]{$P$} (vk-2);
					\draw [dashed, bend left] (r) to node[above]{$C$} (v);
					\draw [dashed] (v2) to (x);
					\draw [dashed] (v2) to (y);
				\end{tikzpicture}
				\caption{Path $P$ and it's neighborhood}
				\label{fig:ColorExchangeLemma}
			\end{figure}
			
			Otherwise, let $C$ be a bichromatic cycle formed because of the exchange (an instance is depicted in Figure~\ref{fig:ColorExchangeLemma}). Since $P$ is maximal, it is clear that $C$ is not an $(\alpha,\beta)$-bichromatic cycle. Note that one of the colors in $C$ is either $\alpha$ or $\beta$ because we did not change any color other than $\alpha$ and $\beta$. This implies that $V(P) \cap V(C) \subseteq \{v_1, v_2, v_{k-1}, v_k\}$ since the alternate degree 2 vertices in $P$ see only the colors $\alpha$ and $\beta$. Therefore, the cycle $C$ may contain only the first or the last edge of the path $P$. The cycle $C$ can not contain the first edge of $P$ since the edge $v_1s$ is not assigned any color in the coloring $d$ as well as in $c'$. Therefore, we can infer that the cycle contains only the last edge (which is $v_{k-1}v_k$) of $P$ with $V(P) \cap V(C) = \{v_{k-1}, v_k\}$. This also implies that $deg(v_{k-1}) \ge 3$ and $deg(v_k)=2$. We can also infer that $C$ is the unique bichromatic cycle formed by the color exchange in $P$ because any cycle formed has to contain the vertex $v_k$ and $deg(v_k)=2$.
			
			Let $r$ be the neighbor of $v_{k-1}$ in the cycle $C$ but not in the path $P$. Clearly $deg(r)=2$ and $deg(v_{k-2})=2$. It is easy to see that $d(v_{k-1}v_k)=\beta$. Hence, $c'(v_{k-1}v_k)=\alpha$. Note that $F_{v_{k-1}v_{k-2}}=\{\alpha\}$. Let $c'(v_{k-1}r)=\gamma$. Therefore, $C$ is an $(\alpha,\gamma)$-bichromatic cycle. Now swap the colors with respect to the edges $v_{k-1}r$ and $v_{k-1}v_{k-2}$ to get a coloring $c''$ which clearly kills the bichromatic cycle $C$. Any new bichromatic cycle $C'$ formed should involve at least one of the edges $\{v_{k-1}r,v_{k-1}v_{k-2}\}$ and hence at least one of the vertices $\{r,v_{k-2}\}$. Note that since $F_r=\{\alpha,\beta\}$ and $F_{v_{k-2}}=\{\alpha,\gamma\}$, we can infer that the cycle $C'$ contains $\alpha$ as one of the colors. Since $c''(v_{k-1}v_k)=\alpha$, the cycle $C'$ contains the vertex $v_k$. But $F_{v_k}=\{\alpha,\gamma\}$, implying that $C'$ is an $(\alpha,\gamma)$-bichromatic cycle. This is a contradiction since the $(\alpha,\gamma)$-bichromatic path from $v_{k-1}$ going towards the vertex $v_k$ ends at $r$. Hence, the coloring $c''$ is a valid partial coloring of $G$.
			
			Now, we need to assign a color to the edge $v_1s$. Note that either $c''(v_1v_2)=\beta$ or $c''(v_1v_2)=\gamma$ with $k=3$. If $\alpha \notin F_{v_1s}$, then we assign the color $\alpha$ to the edge $v_1s$ to get the coloring $c$. Notice that $F_{sv_1}$ is either $\{\beta\}$ or $\{\gamma\}$. If $F_{sv_1}=\{\beta\}$, the only possible bichromatic cycle with respect to the coloring $c$ is an $(\alpha,\beta)$-cycle. But we know that the path $P$ ends at $v_k$. In the coloring $c''$, the $(\beta,\alpha)$-bichromatic path from $v_1$ would end at $v_k$ or $v_{k-2}$, both of which are not the same as $s$. Thus we conclude that there does not exist a $(\beta,\alpha,v_1s)$-critical path with respect to $c''$. If $F_{sv_1}=\{\gamma\}$, the only possible bichromatic cycle with respect to the coloring $c$ is an $(\alpha,\gamma)$-cycle. But we know that in the coloring $c''$, the $(\alpha,\gamma)$-bichromatic path from $v_1$ would go through $v_k$ and end at $r$. Further, since $deg(r)=2$ and $deg(s) \ge 3$, $r \ne s$. Thus we conclude that there does not exist a $(\gamma,\alpha,v_1s)$-critical path with respect to $c''$. Hence, the coloring $c$ is the required coloring.
			
			On the other hand suppose $\alpha \in F_{v_1s}$. Then we have $|F_{v_1s}(d) \cup F_{sv_1}(d)| \le \Delta-1$ and there exists a candidate color $\eta$ for the edge $v_1s$ with respect to the coloring $d$ as well as $c''$. Recall that $\eta \ne \beta$.
			
			Assume that $F_{sv_1}=\{\beta\}$. Suppose the color $\eta$ is not valid for the edge $v_1s$. Then there exists a $(\beta,\eta,v_1s)$-critical path with respect to $c''$. Let this path be $Q = v_1 v_2 u_1 u_2 \cdots s$. Note that $deg(u_1)=2$ and hence $c''(u_1u_2)=\beta$. Recall that there existed an $(\alpha,\eta,v_1s)$-critical path with respect to $d$ which implies $d(u_1u_2)=\alpha$. The only edges that are recolored from $\alpha$ to $\beta$ while obtaining $c''$ from $d$ are the edges in $P$. If $k \le 3$, then it is easy to see that $P$ does not reach the vertex $u_1$. If $k>3$, then observe that any edge colored $\alpha$ in $P$ goes from a degree 2 vertex to a higher degree vertex with respect to the coloring $d$. But $deg(u_1)=2$ implying that $v_k \ne u_1$. Hence, irrespective of the value of $k$, $P$ does not reach the vertex $u_1$. Therefore, $c''(u_1u_2)=\alpha$, a contradiction to the fact that $c''(u_1u_2)=\beta$. Hence, the color $\eta$ is valid for the edge $v_1s$.
			
			Now, we have $F_{sv_1} \ne \{\beta\}$. Then, $F_{sv_1}=\{\gamma\}$ with $k=3$. Note that $c''(v_2r)=\beta$, $F_{v_2r}=\{\alpha\}$ and $c''(v_2v_3)=\alpha$. Recall that there existed an $(\alpha,\gamma)$-bichromatic cycle $C'$ involving the vertices $\{r,v_2,v_3\}$ with respect to the coloring $c'$ which in turn implied that there was an $(\alpha,\gamma)$-bichromatic path starting from $v_1$ going through $r$ ending at $v_3$ with respect to the coloring $d$. Suppose $\eta=\gamma$. Then, since there is no valid color for the edge $v_1s$ with respect to $d$, there exists an $(\alpha,\gamma,v_1s)$-critical path. Therefore, in the coloring $d$, the $(\alpha,\gamma)$-bichromatic path starting from $v_1$ going through $r$ ends at $s$, a contradiction since $v_3 \ne s$. Thus we can infer that $\eta \ne \gamma$. Suppose the color $\eta$ is not valid for the edge $v_1s$. Then there exists a $(\gamma,\eta,v_1s)$-critical path with respect to $c''$. Let this path be $R = v_1 v_2 w_1 w_2 \cdots s$. Note that $deg(w_1)=2$ and hence $c''(w_1w_2)=\gamma$. Recall that there existed an $(\alpha,\eta,v_1s)$-critical path with respect to $d$ which implies $d(w_1w_2)=\alpha$. Note that the only edge that was colored $\alpha$ in $d$ which was recolored to $\gamma$ in $c''$ is the edge $v_1v_2$. Therefore, $c''(w_1w_2)=\alpha$, a contradiction to the fact that $c''(w_1w_2)=\gamma$. Hence, the color $\eta$ is valid for the edge $v_1s$ in $c''$.
			
			Therefore, irrespective of whether $F_{sv_1}=\{\beta\}$ or not, the color $\eta$ is valid for the edge $v_1s$ in $c''$. Hence, we assign $\eta$ to the edge $v_1s$ to get the required valid coloring $c$.
		\end{proof}
		
		Recall that we have a $2$-connected graph $G$ with $\Delta \ge 3$, $\delta \ge 2$ and no edge incident on both degree $2$ vertices. Depending upon whether $G$ is $2$-sparse or not, we have the following cases.
		
		\textbf{Case-i:} $G$ is $2$-sparse.
		
		Consider any edge $xy$ of $G$. Since $G$ is $2$-sparse, either $deg(x)=2$ or $deg(y)=2$. Without loss of generality assume that $deg(x)=2$. Then $3 \le deg(y) \le \Delta$. Since $deg(x)=2$, $|N(x) \setminus y| = 1$. Let $x'$ be the neighbor of $x$ other than $y$. Let $G' = G \setminus xy$. Since $G'$ is chordless and has less than $m$ edges, by I.H., we can acyclically color $G'$ with $\Delta$ colors because $\Delta(G') \le \Delta$. Let $c'$ be an acyclic partial edge coloring of $G$ corresponding to the subgraph $G'$. Let $c'(xx')=\alpha$. Now, we extend $c'$ to a coloring $c$ of $G$.
		
		If $F_{yx} \cap F_{xy} \ne \phi$, then $\alpha \in F_{yx} \cap F_{xy}$. Thus $|F_{yx} \cup F_{xy}| \le \Delta-1$. Hence, there exists a candidate color $\gamma$ for the edge $xy$. If there is no $(\alpha,\gamma,xy)$-critical path, then $\gamma$ is also valid for the edge $xy$ by Lemma~\ref{lem:ColorValidity}. Otherwise there exists an $(\alpha,\gamma,xy)$-critical path, say $P$. Let $y'$ be the neighbor of $y$ along $P$. Then $c'(yy')=\alpha$ and $F_{yy'}=\{\gamma\}$ since $deg(y')=2$. Let $\beta$ be a color other than $\alpha$ and $\gamma$ ($\beta$ exists because $\Delta \ge 3$). Let $Q$ be the $(\alpha,\beta)$-maximal bichromatic path starting from the vertex $x$. Since $F_{yy'}=\{\gamma\}$, $Q$ does not reach $y$ through an edge colored $\alpha$.
		
		If $F_{yx} \cap F_{xy} = \phi$, then there is no $(\alpha,\beta,xy)$-critical path in $G$ for any $\beta$. Hence, if there exists at least one candidate color $\gamma$ for the edge $xy$, then we are done since $\gamma$ is also valid by Lemma~\ref{lem:ColorValidity}. Otherwise no color is a candidate color for the edge $xy$. Thus $|F_{yx} \cup F_{xy}| = \Delta$. Let $\beta$ be a color other than $\alpha$. Let $Q$ be the $(\alpha,\beta)$-maximal bichromatic path starting from the vertex $x$. Since there is no $(\alpha,\beta,xy)$-critical path in $G$, $Q$ does not reach $y$ through an edge colored $\alpha$.
		
		Hence, irrespective of whether $F_{yx} \cap F_{xy} = \phi$ or not, we have an $(\alpha,\beta)$-maximal bichromatic path $Q$ starting from $x$ which does not reach $y$ through an edge colored $\alpha$. If $Q$ reaches $y$ through an edge colored $\beta$, we have an odd cycle in a $2$-sparse graph $G$, a contradiction to Observation \ref{obs:2sparseBipartite}. Thus we can infer that $Q$ does not reach $y$. Since $G$ is $2$-sparse, the alternate vertices in $Q$ are of degree 2 with degree of $x$ being equal to 2 and all the neighbors of any higher degree vertex in $Q$ are of degree 2. Further, we have $N(x)=\{x',y\}$, $deg(y) \ge 3$ and $y \notin V(Q)$. Also since the edge $xy$ is not colored in $c'$, the path $Q$ and the coloring $c'$ satisfy the conditions required by Lemma~\ref{lem:BichromaticPathColorExchange}. Hence, we obtain an acyclic partial edge coloring $c$ of $G$ with a valid color for the edge $xy$ as per Lemma~\ref{lem:BichromaticPathColorExchange}. Since all the edges of $G$ have been colored, $c$ is also an acyclic edge coloring of $G$.
		
		\textbf{Case-ii:} $G$ is not $2$-sparse.
		
		By Lemma~\ref{lem:2-sparseOR2-cutset}, $G$ admits a proper $2$-cutset with a corresponding split. Let $S$ be the set of all splits $(X,Y,a,b)$ of $G$ such that $G_X(a,b)$ is $2$-sparse, is not isomorphic to $K_{2,t}$ for any $t \ge 3$ and in $G_X(a,b)$, $deg(a) \ge 3$ and either $deg(b) \ge 3$ or $deg(b)=2$ with $X$ being minimal. By Lemma~\ref{lem:SpecialSplit}, we know that $S \ne \phi$. If there exists a split $(X,Y,a,b)$ in $S$ such that $deg(a) \ge 3$ and $deg(b) \ge 3$, then consider the split $(X,Y,a,b)$. Otherwise consider a split $(X,Y,a,b)$ in $S$ such that $deg(a) \ge 3$ and $deg(b)=2$ with $X$ being minimal in $S$. Let $x$ be a vertex in $G_X(a,b)$ which is adjacent to $a$ but not $b$ and let $y$ be a vertex in $G_X(a,b)$ which is adjacent to $b$ but not $a$. Since $deg(a) \ge 3$ in $G$ and in $G_X(a,b)$, any neighbor of $a$ in $X$ should have degree $2$ because $G_X(a,b)$ is $2$-sparse. This implies that $deg(x)=2$. By a similar argument, if $b$ has at least two neighbors in $X$ (i.e., if $deg(b) \ge 3$), then any neighbor of $b$ in $X$ (in particular, the vertex $y$) is of degree $2$.
		
		Now consider $G' = G \setminus xa$. Since $G'$ is chordless and has less than $m$ edges, by I.H., we can acyclically color the edges of $G'$ with $\Delta$ colors since $\Delta(G') \le \Delta$. Let $c'$ be an acyclic partial edge coloring of $G$ corresponding to the subgraph $G'$. Since $deg(x)=2$ and $x$ is adjacent to $a$ but not $b$, $x$ should have a neighbor $u$ other than $a$. Let $c'(xu)=\alpha$. Now, we extend $c'$ to a coloring $c$ of $G$. The following claim is useful further down the proof.
		
		\begin{claim}\label{clm:TwoPaths}
			Let $R$ be the $(\alpha,\beta)$-maximal bichromatic path starting from the vertex $x$ and let $T$ be the $(\alpha,\gamma)$-maximal bichromatic path starting from the vertex $x$. Then either $R$ or $T$ does not reach the vertex $b$.
		\end{claim}
		
		\begin{proof}
			By way of contradiction, assume that both $R$ and $T$ reach the vertex $b$. Since $G$ does not have an edge whose end vertices have degree $2$ and $G_X(a,b)$ is $2$-sparse, the alternate vertices in $R$ and $T$ are of degree 2. These paths start from a degree 2 vertex $x$ such that the edge from a degree 2 vertex to a higher degree vertex is colored $\alpha$ and the edge from a higher degree vertex to a degree 2 vertex is colored $\beta$ and $\gamma$ with respect to $R$ and $T$. Suppose $R$ and $T$ reach the vertex $b$ through an edge colored $\alpha$. This implies that $deg(b) \ge 3$. We can infer that any neighbor of $b$ in $X$ is of degree 2. Let $r$ be a neighbor of $b$ in $X$ such that $c'(rb)=\alpha$. Since $r$ is of degree 2, $r$ can see at most one color in $\{\beta,\gamma\}$, a contradiction to our assumption that both $R$ and $T$ reach the vertex $b$. Thus we can infer that $R$ and $T$ can not reach the vertex $b$ through an edge colored $\alpha$ indicating that $R$ and $T$ reach the vertex $b$ through edges colored $\beta$ and $\gamma$ respectively and hence $deg(b) \ge 3$. Since we already have that $deg(a) \ge 3$, $R$ (or $T$) together with the edge $(a,x)$ is an $(a,b)$-path of odd length in $G_X(a,b)$, a contradiction to Observation~\ref{obs:PathLengthEven}. Thus we can infer that either $R$ or $T$ does not reach the vertex $b$.
		\end{proof}
		
		Suppose there is no candidate color for the edge $xa$. Hence $|F_{ax} \cup F_{xa}| = \Delta$, which implies $F_{ax} \cap F_{xa} = \phi$. Let $\beta$ and $\gamma$ be two colors other than $\alpha$. Clearly there is no $(\alpha,\beta,xa)$-critical path and no $(\alpha,\gamma,xa)$-critical path in $G$ since $F_{ax} \cap F_{xa} = \phi$. Let $R$ and $T$ be the $(\alpha,\beta)$-maximal bichromatic path and the $(\alpha,\gamma)$-maximal bichromatic path starting from the vertex $x$. Then by Claim~\ref{clm:TwoPaths}, either $R$ or $T$ does not reach $b$. Without loss of generality assume that $R$ does not reach $b$. Since $R$ does not reach $a$ as well as $b$, $R$ is completely in $G_X(a,b)$ which is $2$-sparse. Note that since $G_X(a,b)$ is $2$-sparse, the alternate vertices in $R$ are of degree 2 with degree of $x$ being equal to 2 and all the neighbors of any higher degree vertex in $R$ are of degree 2. Further, we have $N(x)=\{u,a\}$, $deg(a) \ge 3$ and $a \notin V(R)$. Also since the edge $ax$ is not colored in $c'$, the path $R$ and the coloring $c'$ satisfy the conditions required by Lemma~\ref{lem:BichromaticPathColorExchange}. Hence, we obtain an acyclic partial edge coloring $c$ of $G$ with a valid color for the edge $xy$ as per Lemma~\ref{lem:BichromaticPathColorExchange}. Since all the edges of $G$ have been colored, $c$ is an acyclic edge coloring of $G$.
		
		On the other hand, suppose there exists a candidate color $\gamma$ for the edge $xa$. If there is no $(\alpha,\gamma,xa)$-critical path in $G$, then $\gamma$ is also valid and we are done. Hence, we can assume that there exists an $(\alpha,\gamma,xa)$-critical path $P$ in $G$. Let $v$ be the neighbor of the vertex $a$ along the path $P$. Observe that $c'(av)=\alpha$ and $\gamma \in F_{av}$. Let $\beta$ be a color other than $\alpha$ and $\gamma$. Let $Q$ be the $(\alpha,\beta)$-maximal bichromatic path starting from the vertex $x$.
		
		Assume that $Q$ does not reach the vertex $b$. If $Q$ reaches $a$ through an edge colored $\beta$, then we have an odd cycle in $G_X(a,b)$ which is $2$-sparse, a contradiction to Observation~\ref{obs:2sparseBipartite}. On the other hand, if $Q$ reaches $a$ through an edge colored $\alpha$, then the edge $av \in Q$ which implies $v \in X$. Since $v \in N(a)$, $deg(v)=2$. Therefore, $F_{av}=\{\beta\}$, a contradiction to the fact that $\gamma \in F_{av}$. Thus we can infer that $Q$ does not reach $a$ as well as $b$. Therefore, $Q$ is entirely in $G_X(a,b)$ which is $2$-sparse. Hence, the alternate vertices in $Q$ are of degree 2 with degree of $x$ being equal to 2 and all the neighbors of any higher degree vertex in $Q$ are of degree 2. Further, we have $N(x)=\{u,a\}$, $deg(a) \ge 3$ and $a \notin V(Q)$. Also since the edge $ax$ is not colored in $c'$, the path $Q$ and the coloring $c'$ satisfy the conditions required by Lemma~\ref{lem:BichromaticPathColorExchange}. Hence, we obtain an acyclic partial edge coloring $c$ of $G$ with a valid color for the edge $xy$ as per Lemma~\ref{lem:BichromaticPathColorExchange}. Since all the edges of $G$ have been colored, $c$ is an acyclic edge coloring of $G$.
		
		Now, we assume that $Q$ reaches the vertex $b$. Note that since $P$ is an $(\alpha,\gamma,xa)$-critical path, it is also an $(\alpha,\gamma)$-maximal bichromatic path starting from $x$. Hence by Claim~\ref{clm:TwoPaths}, $P$ does not reach the vertex $b$ implying that $P$ is entirely in $G_X(a,b)$ which is $2$-sparse. Therefore, any edge from $G_Y(a,b)$ incident on the vertex $a$ can not be colored $\alpha$. Let $z$ and $w$ be the successors of the vertex $u$ along $P$ and $Q$ respectively.
		
		Suppose $w \ne b$. Now, we perform a color exchange with respect to the edges $uz$ and $uw$ to obtain a partial coloring $c''$ of $G$ corresponding to the subgraph $G \setminus xa$. Since $F_{uw}=F_{uz}=\{\alpha\}$, the color exchange is valid. Note that by this color exchange we have removed the $(\alpha,\gamma,xa)$-critical path and since $F_{uw}=\{\alpha\}$, there is no new $(\alpha,\gamma,xa)$-critical path formed. Therefore, we can infer that $\gamma$ is a valid color for the edge $xa$ in $c''$. Thus we can obtain an acyclic edge coloring $c$ of $G$ by assigning the color $\gamma$ to the edge $xa$.
		
		On the other hand, suppose $w=b$. This implies that $y=u$, $deg(b)=2$ and $c'(yb)=c'(uw)=\beta$. Thus $y$ is the unique neighbor of $b$ in $X$, implying that $(X',Y',a,y)$ with $X' = X \setminus y$ and $Y' = Y \cup \{b\}$ is a split in $G$. Note that $deg(a) \ge 3$ and $deg(y) \ge 3$. Further, $G_{X'}(a,y)$ is $2$-sparse, since $G_X(a,b)$ is $2$-sparse. Therefore, if $G_{X'}(a,y)$ is not isomorphic to $K_{2,t}$ for any $t \ge 3$, then $(X',Y',a,y) \in S$, a contradiction to the minimality of $X$ in $S$. Hence, $G_{X'}(a,y)$ is isomorphic to $K_{2,t}$ for some $t \ge 3$. Thus we can infer that any neighbor of $y$ in $X'$ is also a neighbor of $a$. Hence, $uz=yz$ is colored $\gamma$ and since $c'(av)=\gamma$ and $v \in X$, we have $z=v$. (refer Figure~\ref{fig:DeltaAboveThree}). Further, the proof is divided into two subcases based on the value of $\Delta(G)$.
		
		\textbf{Subcase-i:} $\Delta(G) \ge 4$.
		
		In this subcase, we have a color $\eta \notin \{\alpha,\beta,\gamma\}$. If no edge incident on the vertex $y$ is colored $\eta$, then we change the color of the edge $xy$ to $\eta$ and assign the color $\gamma$ to the edge $xa$ to get a coloring $c$ of $G$. Since $c(az)=\alpha$, a new $(\gamma,\eta)$-bichromatic cycle is not formed in the coloring $c$. Thus $c$ is the required acyclic edge coloring of $G$. Hence, we can assume that there exists a vertex $k \in N(y) \cap X$ such that $c'(yk)=\eta$. Since any neighbor of $y$ in $X$ is also a neighbor of the vertex $a$, the vertices $k$ and $z$ are adjacent to the vertex $a$. Since $P$ is an $(\alpha,\gamma,xa)$-critical path, $c'(az)=\alpha$. Therefore, $c'(ak) \ne \alpha$. Also, since $\gamma$ is a candidate color for the edge $xa$, $c'(ak) \ne \gamma$ (refer Figure~\ref{fig:DeltaAboveThree}). Now, we perform a color exchange with respect to the edges $yx$ and $yk$ to obtain a partial coloring $c''$ of $G$ corresponding to the subgraph $G \setminus xa$. The coloring $c''$ is proper because $c''(ak) = c'(ak) \ne \alpha$. Further, since $c''(ak) \ne \gamma$, there is no $(\alpha,\gamma,xa)$-critical path with respect to $c''$. Since $c''(az)=\alpha$, there is no new $(\gamma,\eta,xa)$-critical path formed by the color exchange. Therefore, we can infer that $\gamma$ is a valid color for the edge $xa$ in $c''$. Hence, we can obtain an acyclic edge coloring $c$ of $G$ by assigning the color $\gamma$ to the edge $xa$.
		
		\begin{figure}[h]
			\begin{minipage}{0.58\textwidth}
				\centering
				\begin{tikzpicture}
					\begin{scope}[every node/.style={circle,draw,inner sep=0pt,minimum size=2.5ex}]
						\node (k) at (0,2) {$k$};
						\node (z) at (1.5,2)  {$z$};
						\node (x) at (3,2)  {$x$};
						\node (y) at (4,0) {$y$};
						\node (a) at (5.5,4)  {$a$};
						\node (b) at (5.5,0)  {$b$};
					\end{scope}
					\draw (k) to node[below]{$\eta$} (y);
					\draw (z) to node[right]{$\gamma$} (y);
					\draw (x) to node[right]{$\alpha$} (y);
					\draw (k) to (a);
					\draw (z) to node[below]{$\alpha$} (a);
					\draw (x) to (a);
					\draw (y) to node[above]{$\beta$} (b);
					\draw [dotted] (k) to (z);
				\end{tikzpicture}
				\caption{$G[X \cup \{a,b\}]$ when $\Delta \ge 4$}
				\label{fig:DeltaAboveThree}
			\end{minipage}
			\begin{minipage}{0.42\textwidth}
				\centering
				\begin{tikzpicture}
					\begin{scope}[every node/.style={circle,draw,inner sep=0pt,minimum size=2.5ex}]
						\node (v) at (1.5,4) {$v$};
						\node (x) at (1.5,2.5)  {$x$};
						\node (a) at (0,1.5)  {$a$};
						\node (y) at (3,1.5) {$y$};
						\node (p) at (0,0)  {$p$};
						\node (b) at (3,0)  {$b$};
					\end{scope}
					\draw (v) to node[left]{$\beta$} (a);
					\draw (x) to node[below]{$\gamma$} (a);
					\draw (v) to node[right]{$\gamma$} (y);
					\draw (x) to node[below]{$\alpha$} (y);
					\draw (a) to node[right]{$\alpha$} (p);
					\draw (y) to node[left]{$\beta$} (b);
				\end{tikzpicture}
				\caption{$G[X'' \cup \{p,b\}]$ when $\Delta=3$}
				\label{fig:DeltaThree}
			\end{minipage}
		\end{figure}
		
		\textbf{Subcase-ii:} $\Delta(G)=3$.
		
		In this subcase, $G_{X'}(a,y)$ is isomorphic to $K_{2,3}$. Therefore, clearly $v=z$ is the common neighbor of $a$ and $y$ in $X'$ other than $x$. Let $p$ be the unique neighbor of $a$ in $Y'$. Now for this subcase, we will color separately by redefining the split and the coloring. Consider a split $(X'',Y'',p,b)$ where $X''=\{x,v,y,a\}$ and $Y'' = Y \setminus p$. It is easy to see that $G_{X''}(p,b)$ is $2$-sparse, since $G_{X}(a,b)$ is $2$-sparse. Also note that $deg(x)=2$, $deg(v)=2$, $deg(a)=3$ and $deg(y)=3$. Now consider the graph $G''=G_{Y''}(p,b)$. Note that $G''$ is chordless and has less than $m$ edges. Thus by I.H., we obtain an acyclic edge coloring $c''$ of $G''$ with $\Delta$ colors, since $\Delta(G'') \le \Delta$. Let $w$ be the marker vertex in $G_{Y''}(p,b)$. Let $c''(pw)=\alpha$ and $c''(bw)=\beta$. Let $\gamma$ be a color other than $\alpha$ and $\beta$. Now we try to extend $c''$ to a coloring $c$ of $G$. For any edge $e$ in $G'' \setminus w$, assign $c(e)=c''(e)$. Now, assign $c(pa)=c''(pw)=\alpha$, $c(by)=c''(bw)=\beta$, $c(av)=\beta$, $c(yx)=\alpha$ and $c(ax)=c(yv)=\gamma$ (refer Figure \ref{fig:DeltaThree}). It is easy to see that there are no bichromatic cycles in $G$ with respect to the coloring $c$, implying $c$ to be an acyclic edge coloring of $G$.
		
		Therefore, in any case, we are able to color the graph $G$ with $\Delta$ colors, which marks the completion of the proof of Theorem~\ref{thm:ACIChordless}.
	\end{proof}
	
	\section{Algorithm and Complexity Analysis}
	In this section, we provide the sketch of a polynomial time algorithm to acyclically color a chordless graph $G$ with $a'(G)$ colors along with the complexity analysis of the same. As usual, $n$ and $m$ denote the number of vertices and edges of the input graph $G$ respectively. If $\Delta=1$ or $\Delta=2$ with $G$ being acyclic, then $G$ is either a matching or a set of paths. Note that both of these structures can be colored using $a'(G)=\Delta$ colors trivially. If $\Delta=2$ with $G$ being non-acyclic, then each component in $G$ is either a path or a cycle with at least one component being a cycle. Note that a cycle requires three colors and a path requires two colors for acyclic edge coloring. Therefore, $G$ can be colored using $a'(G)=\Delta+1$ colors. Therefore, we can assume that $\Delta \ge 3$. By Remark~\ref{rem:ConnComp}, we assume that $G$ is connected. We can also assume that $G$ is $2$-connected (and hence $\delta(G) \ge 2$), otherwise we can use the linear-time algorithm by \citet{Hopcroft1973Alg4GraphManip} to compute the $2$-connected components of $G$, and the reconstruction of the coloring from the blocks to $G$, is simple. If $G$ has an edge incident on both degree 2 vertices, then we can contract the edge and obtain the coloring of the resultant graph. Note that extending this coloring to the graph $G$ is simple. Hence, we contract all the edges of $G$ that are incident on both degree 2 vertices. Note that the colorings (when $\Delta \le 2$) and all the extensions that are discussed in the paragraph can be done in linear time. Hence, we have a $2$-connected graph $G$ with $\Delta \ge 3$ and no edge incident on both degree 2 vertices.
	
	First check if the graph $G$ is $2$-sparse which can be done in linear time. If $G$ is $2$-sparse, then let $xy$ be an uncolored edge in the $2$-sparse graph. Recall that we have an assumption that $G$ does not have an edge whose end vertices have degree $2$. Thus either $deg(x)=2$ or $deg(y)=2$. Without loss of generality assume that $deg(x)=2$. Now, we obtain a maximal bichromatic path $Q$ starting from $x$ which does not reach $y$ (existence is proved in Case-i of Section~\ref{prf:ACIChordlessThm}) and then obtain the partial acyclic edge coloring of the graph in which the edge $xy$ is also colored, as per the strategy in Lemma~\ref{lem:BichromaticPathColorExchange}. Note that the above mentioned step can be done in linear time $O(m)$ since this is a mere combination of adjacent color checks and recoloring along a path. We repeat this step for each edge in the $2$-sparse graph $G$. Therefore, the coloring of $2$-sparse graph can be done in $O(m^2)$ time.
	
	Otherwise if $G$ is not $2$-sparse, then let $(X,Y,a,b)$ be a split in $G$ satisfying the properties in Lemma~\ref{lem:SpecialSplit}. Now, we choose a neighbor of the vertex $a$, say $x$ and color the edge $ax$ assuming a coloring of $G \setminus ax$. To do that, we obtain a maximal bichromatic path $P$ starting from $x$ which does not reach $a$ and also does not reach $b$ (existence follows from Claim \ref{clm:TwoPaths}) and then obtain the partial acyclic edge coloring of the graph in which the edge $ax$ is also colored, as per the strategy in Lemma~\ref{lem:BichromaticPathColorExchange}. This step (obtaining a color for the edge $ax$) can be done in linear time $O(m)$ since this is a mere combination of adjacent color checks and recoloring along a path.
	
	Now we analyze the time required to get the desired split. We are going to show that such a split can be obtained in $O(n^3)$ time. For each pair of non-adjacent vertices $(a,b)$ in $G$, let $S$ be the set of all components $C$ in $G \setminus \{a,b\}$ that are $2$-sparse in the graph $G[V(C) \cup \{a,b,w\}]$ ($w$ is a marker vertex which is adjacent to both $a$ and $b$). Let $C_1$ be the set of all trivial components in $S$. Let $C_2$ be the set of all non-trivial components in $S$ in which every neighbor of $a$ and every neighbor of $b$ in $G[V(C) \cup \{a,b,w\}]$ is of degree 2. Let $C_3$ be the set of all non-trivial components in $S$ in which there exists some neighbor of $a$ or $b$ in $G[V(C) \cup \{a,b,w\}]$ of degree at least 3.
	
	If $|C_1 \cup C_2| \ge 2$ with $C_2 \ne \phi$, then let $C'$ and $C''$ be any two components in $C_1 \cup C_2$ with at least one of them in $C_2$. Now, assign $X = V(C') \cup V(C'')$ and $Y = V(G) \setminus (X \cup \{a,b\})$. Hence, $(X,Y,a,b)$ is the required split with $deg(a) \ge 3$ and $deg(b) \ge 3$ in $G_X(a,b)$. Otherwise if $|C_1 \cup C_2| < 2$ and $C_2 \ne \phi$, then let $\tilde{C} \in C_2$. If $a$ and $b$ has at least two neighbors in $\tilde{C}$, then assign $X = V(\tilde{C})$ and $Y = V(G) \setminus (X \cup \{a,b\})$ and hence $(X,Y,a,b)$ is the required split with $deg(a) \ge 3$ and $deg(b) \ge 3$ in $G_X(a,b)$. Suppose either $a$ or $b$ has a unique neighbor in $\tilde{C}$. If both $a$ and $b$ have a unique neighbor in $\tilde{C}$, then discard the pair $(a,b)$. Hence, we have that exactly one among $a$ or $b$ has a unique neighbor in $\tilde{C}$. Without loss of generality let $b$ has a unique neighbor $b'$ in $\tilde{C}$. Note that since there is no edge incident on both degree 2 vertices in $G$, $deg(b') \ge 3$. Now, consider the split $(X',Y'a,b')$ where $X' = X \setminus b'$ and $Y' = Y \cup \{b\}$. If $G_{X'}(a,b')$ is not isomorphic to $K_{2,t}$ for any $t \ge 3$, then $(X',Y',a,b')$ is the required split with $deg(a) \ge 3$ and $deg(b') \ge 3$ in $G_{X'}(a,b')$. Otherwise if $G_{X'}(a,b')$ is isomorphic to $K_{2,t}$ for some $t \ge 3$, then $(X,Y,a,b)$ is the required split with $deg(a) \ge 3$ and $deg(b)=2$ with $X$ being minimal.
	
	Otherwise if $C_2=\phi$ and $C_3 \ne \phi$, then let $C$ be a component in $C_3$. If both $a$ and $b$ have unique neighbor in $X$, then discard $C$ and move to the next component in $C_3$, if it exists. Otherwise exactly one of $a$ or $b$ (say $a$) has at least two neighbors in $C$. Let $b''$ be the unique neighbor of $b$ in $C$. Now, assign $X=V(C)$ and $Y = V(G) \setminus (X \cup \{a,b\})$. Note that since there is no edge incident on both degree 2 vertices in $G$, $deg(b'') \ge 3$. Now, consider the split $(X'',Y''a,b'')$ where $X'' = X \setminus b''$ and $Y' = Y \cup \{b\}$. If $G_{X''}(a,b'')$ is not isomorphic to $K_{2,t}$ for any $t \ge 3$, then $(X'',Y'',a,b'')$ is the required split with $deg(a) \ge 3$ and $deg(b'') \ge 3$ in $G_{X''}(a,b'')$. Otherwise if $G_{X''}(a,b'')$ is isomorphic to $K_{2,t}$ for some $t \ge 3$, then $(X,Y,a,b)$ is the required split with $deg(a) \ge 3$, $deg(b)=2$ and $X$ is minimal. If none of the above mentioned conditions hold, then discard the pair $(a,b)$. Continue this for each non-adjacent pair $(a,b)$ in $G$ until we get the desired split satisfying the properties in Lemma~\ref{lem:SpecialSplit}.
	
	Note that this step of verifying whether the given non-adjacent pair $(a,b)$ has a corresponding split in $G$ which is desired, can be done in linear time $O(n+m)$, which is the time required to characterize the components obtained by deleting the vertex pair from $G$. Also, it is easy to see that $m=O(n)$ for a chordless graph $G$. Therefore, since there are $O(n^2)$ such pairs, obtaining a split in $G$ satisfying the properties in Lemma~\ref{lem:SpecialSplit}, can be done in $O(n^3)$ time. It is clear that we need to find at most $n$ splits throughout the algorithm. Hence, the algorithm takes $O(n^4)$ time to find all the splits required throughout the algorithm.
	
	Recall that to color each edge we require $O(m)$ time irrespective of whether the graph is $2$-sparse or not. Thus, to color all the edges we require $O(m^2)$ time. Notice that to extend the coloring of the blocks to the whole graph, we require $O(m)$ time. It is easy to see that we perform at most $O(n)$ such extensions. Therefore, the extensions take at most $O(mn)$ time.
	
	Hence, the running time of the algorithm is $O(n^4+m^2+mn)$ in which the major contributing step is when we repeatedly find a split in $G$ satisfying the properties in Lemma~\ref{lem:SpecialSplit}. Since $m=O(n)$ for a chordless graph, the running time of the algorithm is $O(n^4)$.
	
	\bibliographystyle{plain}
	\bibliography{AcyclicChordless}
	
\end{document}